\documentclass[10pt]{article}
\usepackage{amscd}
\usepackage{amsmath}
\usepackage{amsfonts}
\usepackage{amssymb}
\usepackage{graphicx}
\usepackage{url}
\usepackage{hyperref}
\usepackage{enumerate}
\usepackage{amsthm}		

\newcommand{\matrxinline}[1]{{[#1]}}

\newcommand{\0}{{\bm 0}}
\newcommand{\1}{{\bf 1}}

\newcommand{\A}{\bm{A}}

\newcommand{\B}{\bm{B}}

\newcommand{\Bmatr}[1]{\begin{bmatrix}\displaystyle #1\end{bmatrix}}

\newcommand{\C}{\bm{C}}
\newcommand{\Dc}{{\cal D}}
\newcommand{\D}{\bm{D}}
\newcommand{\Delt}{\bm{\Delta}}

\newcommand{\E}{\bm{E}}

\newcommand{\Fc}{{\cal F}}

\renewcommand{\H}{\bm{H}}

\newcommand{\I}{\bm{I}}

\newcommand{\Pc}{{\cal P}}

\renewcommand{\P}{\bm{P}}

\newcommand{\Q}{\bm{Q}}

\newcommand{\Sc}{{\cal S}}

\newcommand{\Xc}{{\cal X}}
\newcommand{\Yc}{{\cal Y}}

\newcommand{\an}[1]{\begin{align}#1\end{align}}
\newcommand{\ab}[1]{\begin{align*}#1\end{align*}}

\newcommand{\av}{\bm{a}}
\newcommand{\desclist}[1]{\begin{description}#1\end{description}}

\newcommand{\bm}[1]{{\bf #1}}

\newcommand{\bv}{\bm{b}}

\newcommand{\cc}{^{\!*}\,}	

\newcommand{\case}[1]{\left\{ \begin{array}{ll} #1 \end{array}\right.}  

\newcommand{\diaginline}[1]{\mbox{ \bf diag}[{#1}]}

\newcommand{\diagm}[1]{ { \mbox{\bf diag}\matrx{#1} } }

\newcommand{\enumlist}[1]{\begin{enumerate}#1\end{enumerate}}

\newcommand{\eqdef}{:=}

\newcommand{\e}{\bm{e}}

\newcommand{\oldversion}[1]{}

\newcommand{\goesto}{\rightarrow}

\newcommand{\hide}[1]{}	
\newcommand{\temphide}[1]{}	

\newcommand{\matrx}[1]{{\left[ \stackrel{}{#1}\right]}}

\newcommand{\modd}{\!\!\! \mod}

\newcommand{\nbn}{n \times n}
\newcommand{\norm}[1]{{\|#1\|}}

\newcommand{\pr}{\protect}

\newcommand{\ra}{\rightarrow}

\newcommand{\suchthat}{\colon}

\newcommand{\then}{\Longrightarrow} 
\newcommand{\thus}{\then}

\newcommand{\tr}{\!^\top}

\newcommand{\union}{\cup}

\newcommand{\uv}{\bm{u}}

\newcommand{\vv}{\bm{v}}

\newcommand{\x}{\bm{x}}

\newcommand{\y}{\bm{y}}

\newfont{\gilfont}{cmsy10 scaled\magstep0}
\newcommand{\Reals}{\mathbb{R}} 

\newtheorem{Theorem}{Theorem}

\newtheorem{Definition}{Definition}

\newtheorem{Theorem:}[Theorem]{Theorem:}
\newtheorem{Conjecture:}[Theorem]{Conjecture:}
\newtheorem{Corollary:}[Theorem]{Corollary:}
\newtheorem{Definition:}[Theorem]{Definition:}
\newtheorem{Lemma:}[Theorem]{Lemma:}
\newtheorem{Paradox:}[Theorem]{Paradox:}
\newtheorem{Principle:}[Theorem]{Principle:}
\newtheorem{Proposition:}[Theorem]{Proposition:}
\newtheorem{Recursion:}[Theorem]{Recursion:}
\newtheorem{Result:}[Theorem]{Result:}

\newtheorem{Theorem-InOrder}[Theorem]{Theorem}
\newtheorem{Conjecture-InOrder}[Theorem]{Conjecture}
\newtheorem{Corollary-InOrder}[Theorem]{Corollary}
\newtheorem{Definition-InOrder}[Theorem]{Definition}
\newtheorem{Lemma-InOrder}[Theorem]{Lemma}
\newtheorem{Paradox-InOrder}[Theorem]{Paradox}
\newtheorem{Principle-InOrder}[Theorem]{Principle}
\newtheorem{Proposition-InOrder}[Theorem]{Proposition}
\newtheorem{Recursion-InOrder}[Theorem]{Recursion}
\newtheorem{Result-InOrder}[Theorem]{Result}
\newtheorem{Remark-InOrder}[Theorem]{Remark}
\newtheorem{Counterexample-InOrder} [Theorem]{Counterexample}
\newtheorem{Goal-InOrder} [Theorem]{Goal}

\newtheorem{TheoremNum}[Theorem]{Theorem}

\newtheorem{CorollaryNum}[Theorem]{Corollary}

\newtheorem{LemmaNum}[Theorem]{Lemma}

\newtheorem{PropositionNum}[Theorem]{Proposition}

\newcommand{\Remark}[1]{\begin{Remark-InOrder}\protect{\rm #1}\end{Remark-InOrder}}

\newcommand{\AB}{\A^{(1-t)} \circ \B^{(t)}}
\newcommand{\abv}{\av^{(1-t)} \circ \bv^{(t)}}

\newcommand{\citet}{\cite}
\newcommand{\citep}{\cite}
\newcommand{\citealt}{\cite}
\newcommand{\citeauthor}{\cite}
\newcommand{\citeyearpar}{\cite}

\newcommand{\PS}{Property 1}
\newcommand{\eqcc}{\eqref{eq:i=j} and \eqref{eq:i!=j}}
\newcommand{\TFY}{two-fold irreducibility}
\newcommand{\TFI}{two-fold irreducible}

\title{
A Sharpened Condition for Strict Log-Convexity of the Spectral Radius via the Bipartite Graph
}
\author{Lee Altenberg}

\begin{document}
\maketitle
\begin{abstract}
Friedland (1981)\hide{\citet{Friedland:1981}} showed that for a nonnegative square matrix $\A$, the spectral radius $r(e^\D \A)$ is a log-convex functional over the real diagonal matrices $\D$.  He showed that for fully indecomposable $\A$, $\log r(e^\D \A)$ is strictly convex over $\D_1, \D_2$ if and only if $\D_1-\D_2 \neq c \ \I$ for any $c \in \Reals$.  Here the condition of full indecomposability is shown to be replaceable by the weaker condition that $\A$ and $\A\tr\A$ be irreducible, which is the sharpest possible replacement condition.  Irreducibility of both $\A$ and $\A\tr\A$ is shown to be equivalent to irreducibility of $\A^2$ and $\A\tr\A$, which is the condition for a number of strict inequalities on the spectral radius found in Cohen, Friedland, Kato, and Kelly (1982)\hide{\citet{Cohen:Friedland:Kato:and:Kelly:1982}}.  Such `two-fold irreducibility' is equivalent to joint irreducibility of $\A, \A^2, \A\tr\A$, and $\A\A\tr$, or in combinatorial terms, equivalent to the directed graph of $\A$ being strongly connected and the simple bipartite graph of $\A$ being connected. Additional ancillary results are presented.
\end{abstract}
\section{Introduction}
We begin with a theorem of Friedland on the log-convexity of the spectral radius of a nonnegative matrix (`superconvexity' as Kingman \citealt{Kingman:1961:Convexity} called it).
\begin{TheoremNum}[Friedland Theorem 4.2 \protect{\citet{Friedland:1981}}]\label{Theorem:Friedland}
Let $\Dc_n$ be the set of $n \times n$ real-valued diagonal matrices.  Let $r(\A)$ refer to the spectral radius of a matrix $\A$.  Let $\A$ be a fixed $n \times n$  non-negative matrix having a positive spectral radius.  Define $R\suchthat \Dc_n \rightarrow \Reals$ by $R(\D) \eqdef \log r(e^\D \A)$.  
Then $R(\D)$ is a convex functional on $\Dc_n$.  Specifically:  for every $\D_1, \D_2 \in \Dc_n$,
\an{\label{eq:4.17}
R((\D_1+\D_2)/2) \leq (R(\D_1) + R(\D_2)) / 2.
}
Moreover, if $\A$ is irreducible and the diagonal entries of $\A$ are positive (or $\A$ is fully indecomposable) then equality holds in \eqref{eq:4.17} if and only if 
\an{\label{eq:D1D2cI}
\D_1 - \D_2 = c \ \I
}
for some $c \in \Reals$, where $\I$ is the identity matrix.
\end{TheoremNum}

In a recent paper, Cohen \citet{Cohen:2012:Cauchy} asks whether a weaker condition may be substituted in the theorem for the condition that $\A$ be fully indecomposable.  In particular, Cohen asks whether $\A$ being primitive would suffice. 

Here, these questions are answered:  yes --- the condition that $\A$ is fully indecomposable can be weakened; but no --- the condition that $\A$ be primitive is too weak.  A condition is found in between these two that can be substituted in the theorem --- that $\A\tr\A$ be irreducible --- and it will be shown that this condition is the sharpest possible.  The combination of irreducible $\A$ and $\A\tr\A$ is shown to be equivalent to the condition found for several strict inequalities in \citet{Cohen:Friedland:Kato:and:Kelly:1982}, which is that $\A^2$ and $\A\tr\A$ be irreducible.  

Several ancillary results are also presented.  Specific counterexamples are constructed for full indecomposability and primitivity: 1) partly decomposable matrices that nevertheless require $\D_1 - \D_2 = c \ \I$ for equality in \eqref{eq:4.17}, and 2) primitive matrices that produce equality in \eqref{eq:4.17} even though $\D_1 - \D_2 \neq c \ \I$. 


\section{Main Question}
In Theorem \ref {Theorem:Friedland}, the equality in \eqref {eq:4.17} resulting from $\D_1 - \D_2 = c \ \I$ is readily verified for the `if' direction.  What is of interest is therefore the `only if' direction.  This is formalized as follows:
\begin{Definition}[\PS]
An nonnegative $n \times n$ matrix $\A$ is said to \emph{have \PS} when
\enumlist{
\item\  the equality
\an{\label{eq:CDI}
\log r(e^{(1-t) \C + t \D}\A) = (1-t) \log r(e^{\C}\A) + t \log r(e^{\D} \A),
}
for $\C, \D \in \Dc_n$ and some $t \in (0, 1)$, implies $\C - \D$ is a \emph{scalar matrix}, i.e.
\an{\label{eq:CDcI}
\C - \D = c \ \I
}
for some $c \in \Reals$;

\item\  or equivalently, $\C - \D$ being nonscalar implies for all $t \in (0,1)$ that
\an{\label{eq:CDI2}
\log r(e^{(1-t) \C + t \D}\A) < (1-t) \log r(e^{\C}\A) + t \log r(e^{\D} \A).
}
}
\end{Definition}

Irreducibility is central to \PS, so it is now defined.  First some notation needs to be described:
\desclist{
\item[An $n \times n$ matrix] is represented as \pr{$\matrxinline{A_{ij}}_{i,j=1}^n \equiv \matrxinline{A_{ij}}_{i,j \in \{1, \ldots, n\}} \equiv \A$}.
\item[The $(i,j)$ element] of matrix $\A$ is represented by \pr{$[\A]_{ij} \equiv A_{ij}$}.
\item[$\A> \0$, $\x > \0$] means all elements of matrix $\A$ or vector $\x$ are strictly positive.  
\item[$\D \eqdef \diaginline{D_i}$] means a diagonal matrix with diagonal elements $D_i$.
}

Two equivalent properties are typically used to define irreducibility:

\begin{Definition}[Irreducibility, Definition 1 \pr{\citep[p. 50]{Gantmacher:1959vol2}}]
An $n \times n$ square matrix $\A$ is called \emph{irreducible} if the
index set $\{1, 2, \ldots, n\}$ cannot be partitioned into two nonempty sets $\Sc_1, \Sc_2$ such that
\ab{
A_{ij} = 0 \text{ for all } i \in \Sc_1, j \in \Sc_2.
}
\end{Definition}

\begin{Definition}[Irreducibility, Definition 2 \pr{\citep[p. 50]{Gantmacher:1959vol2}}]
An $n \times n$ square matrix $\A$ is called \emph{irreducible} if there is no permutation matrix $\P$ such that 
\ab{
\A = \P \Bmatr{\A_1 & \0_{12} \\ \A_{21} & \A_{2}} \P \tr,
}
where $\P\tr$ is the transpose of $\P$, $\A_{1}$, $\A_{2}$, and $\A_{21}$ are $p \times p$, $q \times q$, and $q \times p$ matrices respectively, $\0_{12}$ is a $p \times q$ matrix of zeros, and $p+q=n$, $p, q \geq 1$. 
\end{Definition}

\begin{Definition}[Reducibility]
A square matrix is called \emph{reducible} if it is not irreducible.
\end{Definition}

For nonnegative matrices, Seneta \citep [Definition 1.6, p. 18]{Seneta:2006}(also used in \citet[p. 61] {Nussbaum:1986:Convexity}) \emph{defines} irreducibility in the following way, but as Gantmacher \citet {Gantmacher:1959vol2} shows, this is a consequence of the definitions above:

\begin{CorollaryNum}[Irreducibility, Nonnegative Matrices \pr{ \citet[Corollary, p. 52] {Gantmacher:1959vol2}}]\label{Corollary:Irreducible}
An $n \times n$ nonnegative matrix $\A$ is {irreducible} if for every pair
$(i, j)$ of its index set, there exists a positive integer $m \equiv m(i, j)$ such that $[\A^m]_{ij} > 0$.
\end{CorollaryNum}

Equivalent to irreducibility is the following key property (usually stated as strong connectivity of the associated directed graph of a matrix, but stated more directly here).

\begin{TheoremNum}[\pr{\citep[Theorem 3.2.1]{Brualdi:and:Ryser:1991}}]\label{Theorem:Path}
A square matrix $\A$ is irreducible if and only if, for each pair of indices $(i, j)$ there is a sequence of nonzero elements from $i$ to $j$,  $(A_{i h_1}, A_{h_1 h_2}, \ldots, A_{h_p, j})$ or $A_{ij} \neq 0$.
\end{TheoremNum}

\section{Results}
We wish to know the properties of $\A$ that are necessary and sufficient to yield \PS.  First it is shown that irreducibility of $\A$ is a necessary condition.

\begin{TheoremNum}\label{Theorem:Reducible}
Reducible nonnegative matrices never have \PS.
\end{TheoremNum}
\begin{proof}
This is established by constructing $\C$ and $\D$ such that $\C-\D$ is nonscalar but \eqref {eq:CDI} holds.

The spectrum of a reducible matrix is the union of the spectra of the irreducible diagonal block matrices of its Frobenius normal form \citep[p. 29-11]{Hogben:2007}.  Its spectral radius is thus the maximum of the spectral radii of these diagonal block matrices.

The Frobenius normal form of a reducible matrix may be represented as a partition of the indices into disjoint nonempty sets $\Fc_1, \ldots, \Fc_\nu$ where $\nu \geq 2$.  So $\Fc_1 \union \Fc_2 \union \cdots \union \Fc_\nu = \{1, 2, \ldots, n\}$.  The irreducible diagonal block matrices are $\A_1, \ldots, \A_\nu$, each of them being principal submatrices of $\A$, where $\A_k \eqdef \matrxinline{A_{ij}}_ {i,j \in \Fc_k}$.
Thus $r(\A) = \max_{k=1, \ldots, \nu} r(\A_k)$.  

In terms of the Frobenius normal form of $e^\D \A$,
\ab{
r(e^\D \A) = \max_{k=1, \ldots, \nu} r(e^{\D_k} \A_k),
}
where $e^{\D_k} \A_k \eqdef [e^{D_i} A_{ij}]_{i,j \in \Fc_k}$.  Let $h$ be one of the maximal blocks, i.e.\ where $r(e^{\D_h} \A_h)=r(e^\D \A)$.  Now, construct $\C$ from $\D$ thus:
\ab{
\case{
\C_h = \D_h + c_h \I_h\\
\C_k = \D_k	& k \neq h.
}
}
where $c_h > 0$.
Clearly $\C - \D$ is not scalar.  For block $h$, \eqref {eq:1.28} becomes
\an{
\log r(\e^{[(1{-}t) \C_h + t \D_h]} \A_h )
&= \log r(e^{[(1{-}t) (\D_h + c_h \I_h) + t \D_h]} \A_h) \notag\\
&= \log r(e^{(1{-}t)  c_h}  \ e^{\D_h} \A_h) 
= (1{-}t)  c_h +  \log r(e^{\D_h} \A_h) \notag\\
&=(1{-}t) ( c_h + \log r(e^{ \D_h  } \A_h)) + t \log r(e^{ \D_h } \A_h) \notag\\
&= (1{-}t) \log r(e^{ (\D_h + c_h \I_h) } \A_h) + t \log r(e^{ \D_h } \A_h) \notag\\
&= (1{-}t) \log r(e^{ \C_h } \A_h) + t \log r(e^{ \D_h } \A_h).\label{eq:BlockH}
}
Thus equality holds in \eqref {eq:1.28} for block $h$.  Since $c_h >0$, for all $k \neq h$ and  $t \in [0,1]$, 
\ab{
r(\e^{[(1{-}t) \C_h + t \D_h]} \A_h ) &= (1{-}t)  c_h +  \log r(e^{\D_h} \A_h)  \geq r(e^{\D_h}\A_h) \geq  r(e^{\D_k}\A_k), 
}
so block $h$ remains a maximal block for all $t \in [0,1]$, hence
\ab{
r(\e^{[(1{-}t) \C_h + t \D_h]} \A_h ) = r(\e^{[(1{-}t) \C + t \D]} \A ).
}
Thus, the equality \eqref{eq:BlockH} implies the equality \eqref{eq:CDI}.  Since \eqref{eq:CDI} holds even though $\C-\D$ is nonscalar, $\A$ does not have \PS.
\end{proof}

The principal tool to be used next is the set of general necessary and sufficient conditions found by Nussbaum \citet[Theorem 1.1, pp. 63--68]{Nussbaum:1986:Convexity} for strict log-convexity of the spectral radius of irreducible nonnegative matrices over certain forms of variation.   Nussbaum \citeyearpar[Remark 1.2, pp. 69--70]{Nussbaum:1986:Convexity} 
applies these methods to the particular case of Theorem 4.2 of \citet{Friedland:1981}. 

So as to be self-contained, relevant excerpts are presented here of Nussbaum's Theorem 1.1 \citep{Nussbaum:1986:Convexity}, which subsumes the theorems in \citet{Kingman:1961:Convexity}, \citet{Cohen:1981:Convexity}, and \citet[Theorems 4.1, 4.2]{Friedland:1981}.  The excerpts also include the relevant parts of Nussbaum's proof. 


\begin{TheoremNum}[Nussbaum \citep{Nussbaum:1986:Convexity}, Theorem 1.1 Excerpt]\label{Theorem:NussbaumExcerpt} 
\ 

Let $\A$ and $\B$ be nonnegative irreducible $n \times n$ matrices.  Let $\av$ and $\bv$ be the Perron vectors of $\A$ and $\B$, so $\A \av = r(\A) \av$ and $\B \bv = r(\B) \bv$.  Let $\D_\av$ and $\D_\bv$ refer to the diagonal matrices whose diagonal elements are from the vectors $\av$ and $\bv$, respectively.

Define the following `log-convex combinations': the $n \times n$ matrix
\an{\label{eq:AB}
 \AB \eqdef [A_{ij}^{1-t} B_{ij}^t]_{i,j=1}^n.
}
and the $n$-vector
\an{\label{eq:ab}
\abv \eqdef [a_{ij}^{1-t} b_{ij}^t]_{i=1}^n.
}

Then for all $t \in [0,1]$,
\ab{
r\bigl(\AB \bigr) \leq r(\A)^{1-t} \ r(\B)^t,
}
with equality for some $t \in (0,1)$ if and only if
\an{\label{eq:EAE}
\B =  \frac{r(\B)}{r(\A)} \E^{-1} \A \E,
}
where $\E \eqdef \D_\av \D_\bv^{-1}$, and in this case equality holds for all $t \in [0,1]$.
\end{TheoremNum}

\begin{proof}[Nussbaum's proof]
The product of the log-convex combinations $\AB$ and $\abv$ manifests H\"{o}lder's inequality.  For each $i = 1, \ldots, n$:
\an{
\sum_{j=1}^n (A_{ij} a_j)^{1-t} (B_{ij} b_j)^t 
& \leq 
\biggl(\sum_{j=1}^n A_{ij} a_j \biggr)^{1-t} 
\biggl(\sum_{j=1}^n B_{ij} b_j \biggr)^{t} \label{eq:HolderSum}\\
&= ( r(\A) \, a_i)^{1-t} \,  (r(\B) \, b_i)^{t}, \notag
}
or, in vector form,
\an{\label{eq:ABab}
(\AB) (\abv) 
&\leq  r(\A) ^{1-t}r(\B)^{t} \ (\abv),
}
with equality for some $t \in (0,1)$ if and only if, for each $i$, the terms in each sum on the right of \eqref {eq:HolderSum} are proportional, i.e. there exists $\gamma_i$ such that
\an{\label{eq:Requirement}
B_{ij} b_j = \gamma_i A_{ij} a_j, \qquad j = 1, \ldots, n.
}
Summation over $j$ in \eqref {eq:Requirement} gives
\ab{
\sum_{j=1}^n B_{ij} b_j = r(\B) \, b_i
= \gamma_i \sum_{j=1}^n A_{ij} a_j = \gamma_i \, r(\A) \, a_i,
}
hence $\gamma_i$ is solved:
\ab{
\gamma_i = \frac{r(\B)}{r(\A)} \frac{b_i}{a_i}.
}
With this, the equality conditions \eqref {eq:Requirement} can be rewritten as
\an{\label{eq:Bequality}
B_{ij} =  \frac{r(\B)}{r(\A)} \frac{b_i}{a_i} A_{ij} \frac{a_j}{b_j}, \qquad i, j = 1, \ldots, n,
}
which is the derivation for \eqref{eq:EAE}.  

The desired term $r(\AB)$ emerges from application of the Subinvariance theorem to \eqref {eq:ABab}.

\begin{TheoremNum}[Subinvariance \pr{\citep [Theorem 1.6, p. 23] {Seneta:2006}}]
For any irreducible nonnegative matrix $\H$ and nonnegative vector $\y$, if $\H \y \leq s \y$, then $\y > \0$ and $r(\H) \leq s$, with equality if and only if $\H \y = r(\H) \y$.  
\end{TheoremNum}
Here,  
$\H = \AB$, $\y = \abv$, and $s =  r(\A)^{1-t} \, r(\B)^{t}$.  Therefore
\ab{
r(\AB) \leq  r(\A)^{1-t} \, r(\B)^{t},
}
with equality for some $t \in (0,1)$ if and only if 
\ab{
(\AB)\ (\abv) =  r(\A)^{1-t} \, r(\B)^{t}\  (\abv)
}
\hide{
\sum_{j=1}^n (A_{ij}^{1-t}B_{ij}^t) \  ( a_j^{1-t}  b_j^t) 
= r(\A)^{1-t} \, r(\B)^{t} \,  ( a_i^{1-t} \,   b_i^{t}),
}
which is precisely equality in \eqref {eq:HolderSum}, whose conditions are given by \eqref {eq:Bequality}, in which case equality holds for all $t \in [0,1]$. 
\end{proof}

If we let $\A$ and $\B$ in Theorem \ref {Theorem:NussbaumExcerpt} be substituted by matrices $e^\C \A$ and $e^\D \A$ we obtain:
\begin{CorollaryNum}[Nussbaum's Remark 1.2 \citet{Nussbaum:1986:Convexity}]
\ 

Let $\A$ be an $\nbn$ irreducible nonnegative matrix, and $\C, \D \in \Dc_n$ be diagonal matrices.  Then for all $t \in [0, 1]$, 
\an{\label{eq:1.28}
\log r(e^{[(1{-}t) \C + t \D]} \A) &\leq (1{-}t) \log r(e^{ \C } \A) + t \log r(e^{ \D } \A),
}
with equality for some $t \in (0, 1)$ if and only if there exists a positive diagonal matrix $\E \in \Dc_n$, and $\alpha > 0$, such that
\an{\label{eq:1.29}
e^\D \A 
= \alpha \, \E^{-1} e^\C \A \E,
}
or in terms of matrix elements,
\an{\label{eq:EqualityConditions}
e^{D_{i}}A_{ij} =  \alpha E_i^{-1} e^{C_{i}} A_{ij} E_j, \qquad  i,j = 1, \ldots, n .
}
\end{CorollaryNum}

With this machinery in place, we are ready to analyze \PS.
Define $L_i \eqdef \log E_i$ and $\Delt \eqdef \D - \C$, i.e.\ $\Delta_i \eqdef D_i - C_i$.  Then \eqref{eq:EqualityConditions} is equivalent to the condition that for each $i,j \in 1, \ldots, n$, 
\an{
A_{ii} &= 0, \text{ or } 
\Delta_i =\log \alpha & j=i \label{eq:i=j};\\
A_{ij} &= 0, \text{ or } 
 \Delta_i = \log \alpha + L_j -L_i & j\neq i. \label{eq:i!=j}
}
For $\A$ to have \PS, satisfaction of the set of equalities \eqref{eq:i=j} and \eqref{eq:i!=j} must imply that $\Delt = c \ \I$.  

What are necessary and sufficient conditions on $\A$ for \eqref{eq:i=j} and \eqref{eq:i!=j} to imply $\Delt=c \ \I$?  We proceed in stages.

\begin{LemmaNum}
\PS\ depends solely upon the sign pattern of $\A$.
\end{LemmaNum}
\begin{proof}
Whenever $A_{ij} >0$, $A_{ij}$ cancels out from both sides of \eqref{eq:EqualityConditions}, so only the sign of $A_{ij}$ (by hypothesis constrained to $0$ or $+$) enters into \eqref{eq:i=j} and \eqref{eq:i!=j}.
\end{proof}

\begin{LemmaNum}\label{Lemma:LiLj}
For irreducible $\A$, if the equality conditions \eqref{eq:i=j} and \eqref{eq:i!=j} are met, and some $L_i \neq L_j$, then $\Delt \neq c \ \I$ for any $c \in \Reals$.
\end{LemmaNum}
\begin{proof}
Suppose to the contrary $\Delt= c \ \I$.  This will be shown to imply that $L_i=L_j$ for all $i,j$.  

Irreducibility of $\A$ means by Theorem \ref{Theorem:Path} that for any pair $i,j \in \{1, \ldots, n\}$ either $A_{ij}>0$, or there is a path of positive elements $(A_{i  \, h_1}, A_{h_1 h_2}, \ldots, A_{h_p j})$, or both.  When $A_{ij} >0$ then \eqref{eq:i!=j} yields 
$\Delta_i = \log \alpha + L_{j} - L_i$,
and when $A_{ij}=0$,  
repeated application of \eqref{eq:i!=j} to the path $(A_{i  \, h_1}, A_{h_1 h_2}, \ldots, A_{h_p j})$ gives:
\an{
\Delta_{i\ \,} &= \log \alpha + L_{h_1} - L_i, \notag \\
\Delta_{h_1} &= \log \alpha + L_{h_2} - L_{h_1}, \notag\\
&\ldots \label{eq:Repeated}\\
\Delta_{h_p} &= \log \alpha + L_{j} - L_{h_p}. \notag
}
Summing them and applying the the hypothesis $\Delta_i = c$ for all $i$ yields
\an{\label{eq:PathID}
\sum_{k \in \{i, h_1, \ldots, h_p\} } \Delta_k &= (p+1) c = (p+1) \log \alpha + L_{j} - L_{i}.
}
The case where $A_{ij} > 0$ can be accommodated in \eqref {eq:PathID} by letting $p=0$.

Irreducibility also implies there must be a reverse path of positive elements  $(A_{j  \, h'_1}, A_{h'_1 h'_2}, \ldots, A_{h'_p i})$ from $j$ to $i$, yielding
\an{\label{eq:PathIDji}
\sum_{k \in \{j, h'_1, \ldots, h'_{p'}\} } \Delta_k &= (p'+1)c = (p'+1) \log \alpha + L_{j} - L_{i}.
}
Summing \eqref {eq:PathID} and \eqref {eq:PathIDji} yields
\an{\label{eq:cSolution}
(p+p'+2) c &=  (p+p'+2) \log \alpha \iff c = \log \alpha.
}
Substitution of $\Delta_i = c = \log \alpha$ in \eqref{eq:Repeated} gives
\ab{
L_i = L_{h_1} = L_{h_2} = \cdots = L_{h_p} = L_j.
}
so $L_i=L_j$.  Since this must hold for every choice of $i,j \in \{1, \ldots, n\}$, this means that $L_i = L_j$ for all $i, j = 1, \ldots, n$.  By contrapositive inference, if some $L_i \neq L_j$, then $\Delt \neq c \ \I$ for any $c \in \Reals$.	
\end{proof}
\ \\

\subsection{Main Results}

\begin{TheoremNum}[Necessary and Sufficient Condition for \PS]\label{Theorem:Main}
For a nonnegative matrix $\A$ to have \PS\ it is necessary and sufficient that $\A$ and $\A\tr\A$ be irreducible.
\end{TheoremNum}
\begin{proof}
Since reducible $\A$ do not have \PS\ by Theorem \ref{Theorem:Reducible}, we assume that $\A$ is irreducible, in which case Nussbaum's \citet[Corollary 1.2 ]{Nussbaum:1986:Convexity} applies, and will be used combinatorially.  We notice that if two non-diagonal elements in a row of $\A$ are positive, e.g. $A_{ij} > 0, A_{ik}>0$, then \eqref{eq:i!=j} gives
\ab{
\Delta_i &= \log \alpha + L_j - L_i  = \log \alpha + L_k - L_i\\ 
\thus \qquad  L_k &= L_j.
}
Thus, equality relations between the $L_i$ variables are the result of a single row having multiple positive elements $A_{ij}>0$.  The identity of the row is irrelevant to the $L_i$ values that are equated.  

This naturally brings us to the bipartite graph associated with $\A$.  Let us define both the directed graph and the simple bipartite graph associated with a matrix.

\begin{Definition}[Associated Directed Graph]
The \emph{directed graph} (also called \emph{digraph}) associated with an $n \times n$ matrix $\A$ consists of a set of $n$ vertices, and a set of directed edges (also called \emph{arcs}), where an edge goes from vertex $j$ to vertex $i$ when $A_{ij} \neq 0$.
\end{Definition}

\begin{Definition}[Associated Bipartite Graph]
The \emph{simple bipartite graph} associated with an $n \times m$ matrix $\A$ consists of a set $\Xc$ of $n$ vertices corresponding the row indices of the matrix, a set $\Yc$ of $m$ vertices corresponding to the column indices, and a set of undirected edges, where an edge goes between $X_i \in \Xc$ and $Y_j \in \Yc$ when $A_{ij} \neq 0$.  
\end{Definition}

Let us return to the situation in which a row of $\A$ has two positive elements, $A_{ij}$ and $A_{ik}$.  In the bipartite graph associated with $\A$, this means that there are edges between vertices $Y_k$ and $X_i$, and between $X_i$ and $Y_j$.  In other words, there is a path between vertices $Y_k$ and $Y_j$ passing through $X_i$.  The existence of a path, and thus equality of $L_j$ and $L_k$, can be conveniently represented as the condition $\sum_{i=1}^n A_{ij} A_{ik}= [\A\tr\A]_{jk} \neq 0$.

The transitivity of equality means that if there is a path of any length between $Y_j$ and $Y_k$ (going back and forth between the $Y_i$'s and the $X_i$'s), then $L_j=L_k$.  This occurs if and only if there is some integer $m \geq 1$ such that $ [(\A\tr\A)^m]_{jk} > 0$.  When there is some such $m_{jk}$ for every $j, k \in \{1, \ldots, n\}$, this makes $\A\tr\A$ irreducible by Corollary \ref{Corollary:Irreducible}.

Therefore, if $\A\tr\A$ is irreducible in addition to $\A$ being irreducible, the equality conditions \eqcc\ imply $L_i \equiv L$ for all $i=1, \ldots, n$ (thus $\E$ in \eqref{eq:1.29} is a scalar matrix), hence $\Delta_i = \log \alpha$ for all $i$, so $\Delt = \D - \C = \log \alpha \ \I$, satisfying \eqref{eq:CDcI}, hence $\A$ has \PS.  The sufficient-part of the theorem is thus proven.

The necessary-part means that if $\A\tr\A$ is reducible, then $\A$ does not have \PS.  This means that the equality conditions \eqref{eq:EqualityConditions} can be met even while $\Delt \neq c \ \I$ for any $c \in \Reals$.  

To show this, let $\A\tr\A$ be reducible.  Then there exist $j$ and $k$ such that $[(\A\tr\A)^m]_{jk} $ $= 0$ for all integers $m \geq 1$.  For this pair $j$ and $k$ there is no $m \geq 1$ that gives positive $[(\A\tr\A)^m]_{jk}$ to imply $L_j =L_k$.  Hence we may set $L_j \neq L_k$ and still meet \eqcc.  From Lemma \ref {Lemma:LiLj}, this implies that $\Delt \neq c \ \I$ for any $c \in \Reals$.  Thus the equality conditions \eqcc\ do not require $\Delt = c \ \I$, so $\A$ does not have \PS.
\end{proof}

Application of Theorem \ref{Theorem:Main} allows Theorem 4.2 of \citet{Friedland:1981} to be sharpened as follows.

\begin{TheoremNum}[{Sharpening of Friedland's Theorem 4.2 \citealt{Friedland:1981}}]\label{Theorem:Sharpening}
\ 

Let $\Dc_n$ be the set of $n \times n$ real-valued diagonal matrices.  Let $\A$ be a fixed $n \times n$  non-negative matrix having a positive spectral radius. 

Then:
\enumlist{
\item\  for every $\C, \D \in \Dc_n$, $t \in (0, 1)$,
\an{\label{eq:ConvexCombination}
\log r(e^{(1-t) \C + t \D}\A) \leq (1-t) \log r(e^{\C}\A) + t \log r(e^{\D} \A);
}
\item\  if $\D_1 - \D_2$ is scalar, equality holds in \eqref{eq:ConvexCombination};
\item\  the following are equivalent:
\enumlist{
\item\  $\A$ and $\A\tr\A$ are irreducible;
\item\  equality holds in \eqref{eq:ConvexCombination} only if $\D_1 - \D_2$ is scalar;
\item\  strict inequality holds in \eqref{eq:ConvexCombination} for all pairs $\D_1, \D_2 \in \Dc_n$ for which $\D_1 - \D_2$ is nonscalar.
}
}
\end{TheoremNum}
The condition that $\A\tr\A$ be irreducible in order to yield strict inequality also arises in Lemmas 3, 4 and 5 of Cohen et al.\ \citet{Cohen:Friedland:Kato:and:Kelly:1982}.  It is notable that they arrive at this condition through analytic means, rather than matrix-combinatorial path used here.  Specifically, $\A\tr\A$ enters through the matrix norm $\norm{\A} \eqdef r(\A\cc\A)$, where the complex conjugate $\A\cc = \A\tr$ when $\A$ is real. 

In their Lemmas 3, 4 and 5, the condition that $\A\tr\A$ be irreducible for strict inequality
is accompanied by the condition that $\A^2$ also be irreducible.  
In Theorem \ref{Theorem:Two-Fold}, next, we shall see that irreducibility of both $\A^2$ and $\A\tr\A$ is equivalent to irreducibility of both $\A$ and $\A\tr\A$.  It may therefore make sense  to call such matrices \emph{two-fold irreducible}. 

\begin{Definition}[Two-fold Irreducibility]
A nonnegative square matrix $\A$ is called \emph{two-fold irreducible} if $\A$ and $\A\tr\A$ are irreducible.
\end{Definition}

Two-fold irreducibility, then is the underlying condition that is necessary and sufficient for strict inequality in Theorem \ref{Theorem:Sharpening} here, and in Lemmas 3, 4 and 5 of \citet{Cohen:Friedland:Kato:and:Kelly:1982}. 

The proof of Theorem \ref{Theorem:Two-Fold} requires the following theorem from \citet {Brualdi:and:Ryser:1991} (restated in  \citet[p. 29-10]{Hogben:2007}), regarding the \emph{index of imprimitivity} or \emph{period} of $\A$, which is the greatest common divisor of the length of all cycles in $\A$.

\begin{TheoremNum}[{Brualdi and Ryser Theorem 3.4.5 \citealt{Brualdi:and:Ryser:1991}}]\label{Theorem:Brualdi+Ryser}
Let $\A$ be an irreducible, nonnegative matrix with index of imprimitivity $\gamma \geq 2$. Let $m$ be a positive integer.  Then $\A^ m $ is irreducible if and only if $m$ and $\gamma$ are relatively prime.
\end{TheoremNum}

Theorem \ref {Theorem:Brualdi+Ryser} will be used in the proof of the following equivalence.
\begin{TheoremNum}[Two-fold Irreducibility]\label{Theorem:Two-Fold}
For a nonnegative square matrix $\A$, the following are equivalent:
\begin{enumerate}[1.]
\item\  $\A$ and $\A\tr\A$ are irreducible;\label{item:A-AtA}
\item\  $\A^2$ and $\A\tr\A$ are irreducible;\label{item:A2-AtA}
\item\  $\A$ and $\A\A\tr$ are irreducible; \label{item:A-AAt}
\item\  $\A^2$ and $\A\A\tr$ are irreducible; \label{item:A2-AAt}
\item\  $\A$, $\A^2$, $\A\tr\A$, and $\A\A\tr$ are irreducible; \label{item:A-AtA-AAt}
\item\   the directed graph of $\A$ is strongly connected and the simple bipartite graph of $\A$ is connected.\label{item:Connected}
\end{enumerate}
\end{TheoremNum}

\begin{proof}
If $\A^2$ is irreducible, then $\A$ is irreducible, so irreducibility of $\A^2$ and $\A\tr\A$ implies irreducibility of $\A$ and $\A\tr\A$, and irreducibility of $\A^2$ and $\A\A\tr$  irreducibility of $\A$ and $\A\A\tr$.
Conversely, suppose that $\A$ is irreducible and $\A^2$ is reducible.  Then by Theorem \ref{Theorem:Brualdi+Ryser}, $2$ and the index of imprimitivity $\gamma$ are not relatively prime, hence $2$ divides $\gamma$, which means $\gamma \geq 2$ and $\A$ is imprimitive.  

Hence $\A$ may be put into the following cyclic normal form by some permutation matrix $\P$ (\citealt[Sec. 2.2]{Berman:and:Plemmons:1994};  \citealt[Sec. 3.4]{Brualdi:and:Ryser:1991}; 
\citealt[Sec. 3.3Ð3.4] {Minc:1988}; restated in \citealt[p. 29-10]{Hogben:2007}):
\ab{
 \A  = \P\tr \Bmatr{
\0 & \B_2 \\
\B_1 & \0
}\P.
}
This yields
\ab{
 \A\tr\A  & =
\P\tr \Bmatr{
\0 & \B_1\tr \\
\B_2\tr & \0
}\P
\P\tr \Bmatr{
\0 & \B_2 \\
\B_1 & \0
}\P
= 
\P\tr \Bmatr{
\B_1\tr \B_1 & \0 \\
 \0& \B_2\tr \B_2
}\P,\\
\intertext{and}
 \A\A\tr  & =
\P\tr \Bmatr{
\0 & \B_2 \\
\B_1 & \0
}\P
\P\tr \Bmatr{
\0 & \B_1\tr \\
\B_2\tr & \0
}\P
= 
\P\tr \Bmatr{
\B_2\B_2\tr  & \0 \\
 \0& \B_1 \B_1\tr
}\P.
}
The presence of the two $\0$ block matrices makes $\A\tr\A$ and $\A\A\tr$ reducible.   Thus for irreducible $\A$, the assumption that $\A\tr\A$ or $\A\A\tr$ are irreducible implies by contrapositive inference that $\A^2$ is irreducible.  Thus far it is shown \ref{item:A-AtA}$\iff$\ref{item:A2-AtA} and \ref{item:A-AAt}$\iff$\ref{item:A2-AAt}.

In the bipartite graph of $\A$, irreducibility of $\A\tr\A$ means that the $\Yc$ vertices are connected.  Irreducibility of $\A$ requires that each row have at least one positive element, and thus each $\Xc$ vertex is connected to the connected $\Yc$ vertices, making the entire bipartite graph connected, in particular the $\Xc$ vertices.  Thus $\A\A\tr$ is irreducible.  Similarly, irreducibility of $\A$ requires that each column have at least one positive element, so combined with irreducibility of $\A\A\tr$, the same argument yields that $\A\tr\A$ is irreducible.  This gives us \ref{item:A-AtA}$\iff$\ref{item:A-AAt}, and ties together in equivalence \ref{item:A-AtA}, \ref{item:A2-AtA}, \ref{item:A-AAt}, \ref{item:A2-AAt}, hence \ref {item:A-AtA-AAt}.  In addition, \ref{item:A-AtA}$\iff$\ref{item:A-AAt}$\implies$\ref{item:Connected}.

Berman and Grone \citet[Lemma 2.1]{Berman:and:Grone:1988} show that the bipartite graph of $\A$ is connected if and only if $\A\A\tr$ and $\A\tr\A$ are irreducible.  Boche and Stanczak \citet[Theorem 3] {Boche:and:Stanczak:2008:Strict} show that for irreducible $\A$ the bipartite graph of $\A$ is connected if and only if $\A\A\tr$ is irreducible.
This combined with connectedness of the directed graph of $\A$ gives us \ref{item:Connected}$\implies$\ref{item:A-AtA}, \ref{item:A-AAt}.  The equivalence of all the statements is thus shown.
\end{proof}
\Remark{
Shmuel Friedland (personal communication) conjectured \ref{item:A-AtA}$\implies$\ref{item:A-AAt}.   Joel E. Cohen (personal communication) pointed out that a \emph{scrambling matrix} \citep{Hajnal:1958:Weak} will have irreducible $\A\A\tr$, since by definition, a scrambling has some $k$ such that $A_{ik}A_{jk} > 0$ for every pair $i \neq j$, so this makes $\A\A\tr$ strictly positive off the diagonal, hence irreducible. 
}
\Remark{
It should be noted that irreducibility of $\A\tr\A$ and $\A\A\tr$ does not imply irreducibility of $\A$, as can be seen with the simplest example,
\ab{
\A = \Bmatr{1&0\\1&1}, \A\tr\A= \Bmatr{2&1\\1&1}, \A\A\tr = \Bmatr{1&1\\1&2}.
}
}	

\Remark{
In a model of wireless network reception, Boche and Stanczak \citet{Boche:and:Stanczak:2008:Strict} come close to stating results the same as Theorem \ref {Theorem:Sharpening}.  Their Theorem 1 is the same as \citet[Theorem 4.2]{Friedland:1981} except that conditions for strict convexity are not addressed.  Strict convexity conditions are sought not for the spectral radius itself, but for the shape of regions $\Fc$ in $\Dc_n$ that yield a bounded spectral radius,
\ab{
\Fc(\A) \eqdef \{\D\suchthat r(e^\D \A) \leq 1\} \subset \Dc_n,
}
which comprise the feasible solutions to their signal-to-interference ratio optimization problem.  
The diagonal matrices $\D$ that they consider \citet[Appendix A, p. 1516]{Boche:and:Stanczak:2008:Strict} are not entirely general, but fall within a set $\D \in \Sc(\A)$ derived as
\ab{
 \Sc(\A) \eqdef \left\{\diagm{\log \frac{x_i}{ [\A\x]_i}}_{i=1}^n  \suchthat \x > \0 \right\} \subset \Dc_n.
}

They show that regions $\Fc(\A)$ are convex, and strictly convex provided $\A$ and $\A\A\tr$ are irreducible.  The core of their proof \citep[Eq. (10), (11), and Appendix B, p. 1517] {Boche:and:Stanczak:2008:Strict} utilizes the Cauchy-Schwartz inequality on an expression essentially the same as \eqref {eq:HolderSum}, but with exponent $t = 1/2$.
}	

\Remark{
The product $\A\tr\A$ plays an important role in Markov chains.  Kantrowitz et al.\ \citet[Theorem 2.1] {Kantrowitz:Neumann:and:Ransford:2011} show that a column stochastic matrix $\A$ is \emph{contractive} if and only if $\A\tr \A > \0$;  by contractive they mean $\norm{\A \uv - \A \vv}_1 < \norm{\uv - \vv}_1$ for all $\uv \neq \vv$,  $\uv, \vv \in \Pc_n$, where $\norm{\uv}_1 \eqdef \sum_{i=1}^n |u_i|$, and $\Pc_n$ is the set of probability vectors.  Further, they show that $\A^m \x \goesto \vv$ as $m \goesto \infty$ for all $\x \in \Pc_n$ and some $\vv \in \Pc_n$ if and only if $\A$ is `eventually scrambling' (my phrase by analogy with `eventually positive'), i.e.\ there is some integer $m$ such that $(\A^m)\tr \A^m > \0$ \citet[Theorem 2.3] {Kantrowitz:Neumann:and:Ransford:2011}.  It is notable that the product $\A\A\tr$ does not enter into these results.
}
\subsection{Ancillary Results}
The paper is concluded with a number of additional results.
\begin{PropositionNum}\label{Prop:Monotonicity}
Two-fold irreduciblity is monotonic in the sign pattern of a nonnegative matrix $\A$, i.e.\ if $\A$ has \TFY, then changing an element of $\A$ from $0$ to a positive value maintains \TFY.  
\end{PropositionNum}
\begin{proof}
This is immediate since the sign pattern of $\A\tr\A$ is monotonic in the sign pattern of $\A$, and irreducibility is monotonic in the sign pattern of a nonnegative matrix.  

It can also be seen through direct examination of \eqref{eq:i!=j}.  From Theorem \ref {Theorem:Main}, \TFY\ and therefore \PS\ implies for equality condition \eqref{eq:i!=j} that $L_i=L_j$ for all $i,j$.  Suppose that $A_{ij}=0$ and we change it to be $A_{ij}>0$.  This adds a new constraint to the equality conditions that $\Delta_i = \log \alpha + L_j - L_i$.  However this equation is already satisfied when $\A$ has \PS, which it does by hypothesis, and so the additional equation has no effect.
\end{proof}

\begin{PropositionNum}\label{Prop:FullIndecomp}
Full indecomposability is sufficient but not necessary for \TFY.
\end{PropositionNum}
\begin{proof}
Friedland \citet[Theorem 4.2]{Friedland:1981} proved that full indecomposability is sufficient for \PS.  To prove that it is not necessary I construct a general example of matrices that have \PS, and thus are \TFI, but which are partly decomposable.  A specific example is illustrated or $n=5$.  Without loss of generality $1$ is used for the positive elements.  Let
\an{\label{eq:PartDecompTFI}
\A &= \Bmatr{
0&1&0& 0&1\\
1&0&0&0&0\\
0&1&0&0&0\\
0&0&1&0&0\\
1&1&1&1&0
}.
}
$\A$ contains the $5$-cycle, $5 \to 1 \to 2 \to 3 \to 4 \to 5$ which makes it irreducible.  Application of \eqref{eq:i!=j} to the bottom row's $1$s produces $L_1 = L_2 = L_3 = L_4$.  The top row gives $L_2=L_5$.  Thus all $L_i$ are equal, so $\Delt = \log \alpha \ \I$.  It is easily verified that $(\A\tr\A)^2 > \0$, hence  $\A\tr\A$ is irreducible.

$\A$ is partly decomposable, however, as can be seen from the permutation matrix $\Q$ that rotates the rows up by one, since $\Q\A$ has $\0$-submatrices of size $k$ by $5 -k $ for each $k= 1, 2,3, 4$:
\ab{
\Q\A =\Bmatr{
1&0&0&0&0\\
0&1&0&0&0\\
0&0&1&0&0\\
1&1&1&1&0\\
0&1&0& 0&1
}.
}
This example can be extended to any $n$ as follows:
\an{\label{eq:PSDecomposable}
\case{
A_{1+(j \modd n) , j} > 0, & j=1, \ldots, n\\
A_{ni} > 0, & \forall i \neq n\\
A_{1,2}  > 0, \\
A_{ij}=0, & \text{otherwise.}
}
}
The condition $A_{1+(j \modd n) , j} > 0$ for $j=1, \ldots, n$ produces an $n$-cycle, which makes $\A$ irreducible.  The condition $A_{n,i} > 0 \ \forall i \neq n$ makes all $L_i$ equal for $i=1, \ldots, n{-}1$.  $L_n$ is brought into the equality with the condition $A_{1,2}  > 0$, which in combination with $A_{1,n} > 0$ gives $L_n=L_2$.  Therefore, $L_i = L$ for $i = 1, \ldots, n$.  Substitution in \eqref{eq:i!=j} gives $\Delta_i = \log \alpha$ for $i = 1, \dots, n$.  Therefore the equality conditions \eqref{eq:EqualityConditions} imply $\Delt = \log \alpha \ \I$, so $\A$ has \PS.

To verify that $\A$ constructed according to \eqref{eq:PSDecomposable} is partly decomposable, we note that applying a permutation $\Q$ that rotates the rows of $\A$ upward by 1 satisfies $[\Q\A]_{i,n}=0, i=1, \ldots, n{-}1$.  This is a $n-1$ by $1$ submatrix of zeros, making $\A$ partly decomposable.
\end{proof}

\begin{PropositionNum}\label{Proposition:Primitivity}
Primitivity is necessary but not sufficient for \TFY.
\end{PropositionNum}
\begin{proof} \ 

\emph{Primitivity is necessary for \TFY.}  If $\A$ is irreducible but imprimitive, then there exists a permutation matrix $\P$ such that $\P \A \P\tr$ is in cyclic normal form:
\an{\label{eq:CyclicNormal}
\A  
&=\P\tr \Bmatr{
\0&\0& \cdots& \0&\0&\B_\gamma\\
\B_1&\0& \cdots&\0&\0&\0\\
\0&\B_2& \cdots&\0&\0&\0\\
&&&\cdots&& \\
\0&\0&\cdots&\B_{\gamma-2}&\0&\0\\
\0&\0& \cdots&\0&\B_{\gamma-1} &\0
}\P ,	
}
where each $\0$ block along the diagonal is a square matrix of zeros, of possibly different orders, while the $\B_h$ and $\0$ blocks off the diagonal are rectangular matrices, and $\gamma$ is the index of imprimitivity of $\A$ \citealt[p. 29-10]{Hogben:2007}.  This yields
\ab{
\A\tr\A  
&=\P\tr \Bmatr{
\B_1\tr\B_1&\0& \cdots&\0&\0&\0\\
\0&\B_2\tr\B_2& \cdots&\0&\0&\0\\
&&&\cdots&& \\
\0&\0&\cdots&\B_{\gamma-2}\tr\B_{\gamma-2}&\0&\0\\
\0&\0& \cdots&\0&\B_{\gamma-1}\tr\B_{\gamma-1} &\0\\
\0&\0& \cdots& \0&\0&\B_\gamma\tr\B_\gamma\\
}\P ,	
}
which shows $\A\tr\A$ to be reducible.  Therefore primitivity is necessary for \TFY. 

\emph{Primitivity is not sufficient for \TFY.} To show this, a general example is provided by the \emph{Wielandt matrix} \citep{Kantrowitz:Neumann:and:Ransford:2011}.  An $\nbn$ primitive matrix $\A$ is generated by taking an $n$-cycle graph and adding a shortcut edge so that it gains a circuit of length $n-1$.  Since the greatest common factor of $n$ and $n-1$ is 1, the adjacency matrix for this strongly connected directed graph is aperiodic, hence it is primitive.  The matrix has $n+1$ positive elements.  It is specified by 
\ab{
A_{3,1} & =1, \qquad A_{ij} = \delta_{i, \, 1+(j \modd n)}   \text{ otherwise.}
}
In the $n$th row, for $j = 1, \ldots, n$,
\ab{
[\A\tr\A]_{nj} &= \sum_{k=1}^n A_{kn} A_{kj}
= \sum_{k=1}^n \delta_{k, \, 1+(n \modd n)} \, \delta_{k, \, 1+(j \modd n)} \\
&= \sum_{k=1}^n \delta_{k, \, 1} \, \delta_{k, \, 1+(j \modd n)} 
=  \delta_{1, \, 1+(j \modd n)} 
= \delta_{jn}  ,
}
so row $n$ has a $1$ by $n - 1$ submatrix of zeros, making $\A\tr\A$ reducible.  

To show a concrete example, we add the shortcut $1 \to 3$ to the cycle $1\to 2 \to 3 \cdots \to 5 \to 1$:
\ab{
\A &= \Bmatr{
0&0&0& 0&1\\
1&0&0&0&0\\
1&1&0&0&0\\
0&0&1&0&0\\
0&0&0&1&0
}, 
\text{ giving }
\A\tr\A = \Bmatr{
2&1&0& 0&0\\
1&1&0&0&0\\
0&0&1&0&0\\
0&0&0&1&0\\
0&0&0&0&1
},
}
which is reducible.  To summarize, primitivity of $\A$ is necessary but not sufficient to provide \TFY.
\end{proof}

Proposition \ref {Proposition:Primitivity} can be further illustrated with a completely worked-out example.  Consider a stochastic matrix whose graph comprises two cycles: $1 \ra 2 \ra 1$ and $1 \ra 3 \ra 4 \ra 1$:
\ab{
\A = \Bmatr{0 & 1 & 0 & 1\\ 1/2 & 0&0&0\\1/2 & 0&0&0\\ 0&0&1&0}
}
We confirm $\A$ is primitive by noting that the index of primitivity is the product of the cycle periods $2$ and $3$ \citep[p. 9-7]{Hogben:2007}, $2 * 3 = 6$, and $\A^6 > \0$:
\ab{
\A^6 = \Bmatr{3/8 & 1/2 & 1/4 & 1/2 \\ 1/4 & 1/8 & 1/4 &1/8 \\ 1/4 & 1/8 & 1/4 &1/8 \\ 1/8 & 1/4 &1/4 & 1/4}.
}
The equality condition \eqref{eq:i!=j} for $A_{ij} \neq 0$, $ j\neq i$, is
$
D_{i}- C_{i}= \log \alpha + L_j -L_i .
$
The five nonzero entries of $\A$ ($A_{12}, A_{14}, A_{21}, A_{31}, A_{43}$) thus
give five constraints as the equality conditions:
\ab{
A_{12} \suchthat \qquad   D_{1}- C_{1} &= \Delta_1 = \log \alpha + L_2 - L_1,\\
A_{14} \suchthat \qquad D_{1}- C_{1} &= \Delta_1 = \log \alpha + L_4 - L_1, \\
A_{21} \suchthat \qquad D_{2}- C_{2} &= \Delta_2 = \log \alpha + L_1 - L_2, \\
A_{31} \suchthat \qquad D_{3}- C_{3} &= \Delta_3 = \log \alpha + L_1 - L_3,\\
A_{43} \suchthat \qquad D_{4}- C_{4} &= \Delta_4 = \log \alpha + L_3 - L_4.
}
Since there are $5$ constraints on $4$ variables $\Delta_i$, $4$ variables $L_i$, and variable $c$, there are at least $4+4+1-5=4$ degrees of freedom in any solution.  The above system reduces to:
\ab{
L_4 &= L_2, \\
\Delta_1 &= \log \alpha + L_2 - L_1, \\
\Delta_2 &= \log \alpha + L_1 - L_2, \\
\Delta_3 &=  \log \alpha + L_1 - L_3,\\
\Delta_4 &= \log \alpha + L_3 - L_2.
}
Thus we are free to specify $\alpha, L_1, L_2, L_3$, which includes values that make $\Delt = \D - \C$ nonscalar.

As a concrete example of nonscalar $\Delt$, let $\alpha = e^3$, $L_1 =  1$, $L_2 = -1$, and $L_3 = 2$.  Let $C_i = 0$ for $i=1,2,3,4$, so $D_i = \Delta_i$.  Then $\D = \Delt = \diaginline{(1,5,2,6)}\neq c \ \I$ for any $c \in \Reals$.
\hide{
\D = \Bmatr{
1 & 0 & 0 & 0 \\
0 & 5 & 0 & 0 \\
0 & 0 &2 & 0 \\
0 & 0 & 0 & 6 }.
}
The equality condition is met in \eqref{eq:4.17} if
\ab{
\phi(t) &\eqdef  (1{-}t) \log r(e^{ \C } \A) + t \log r(e^{ \D } \A) - \log r(e^{(1{-}t) \C + t \D} \A)  = 0.
}
Since $\C = \0$, this simplifies to:
\ab{
\phi(t) &=  t \log r(e^{ \D } \A) - \log r(e^{ t \, \D} \A) .
} 
It is readily verified that $\log r(e^{ \D } \A) = 3$ and $ \log r(e^{ t \, \D} \A) = 3 t$, hence $\phi(t)  = 0$.  Thus the equality condition is met in \eqref{eq:4.17} while $\A$ is primitive and $\D - \C \neq c \ \I$ for any $c \in \Reals$.
 
\begin{PropositionNum}\label{Proposition:TIF-Primitive}
A sign-symmetric nonnegative matrix is \TFI\ if and only if it is primitive. 
\end{PropositionNum}
\begin{proof}
By Proposition \ref {Proposition:Primitivity}, primitivity is necessary for \TFY.  It remains to be proven that primitivity is sufficient here.  A matrix $\B$ is \TFI\ if and only if its  sign pattern matrix, $\A$, is \TFI.  A sign-symmetric matrix $\B$ by definition yields $\A = \A\tr$.  Assume $\A$ to be symmetric and primitive.  Primitivity means there is some integer $m\geq 1$ such that $\A^m > \0$ \citep[Theorem 8, p. 80]{Gantmacher:1959vol2}, so clearly $\A^{2m} > \0$, hence $\A^2$ is primitive as well \citep[Corollary 1, p. 82]{Gantmacher:1959vol2}.   Thus $\A\tr \A = \A^2$ is irreducible, hence $\B\tr\B$ is irreducible, making $\A$ and $\B$ \TFI.
\end{proof}

\begin{PropositionNum}\label{Proposition:PrimitiveOrBipartite}
An irreducible sign-symmetric nonnegative matrix is either primitive or cyclic of period $2$.
\end{PropositionNum}
\begin{proof}
Let $\A$ be the sign pattern matrix for the matrix $\B$.  Clearly both $\B$ and $\A$ share the same properties with respect to being irreducible, primitive, or cyclic.  Since $\A$ is irreducible, either $\A$ is primitive, or it is imprimitive with index of imprimitivity $\gamma \geq 2$.  By \citealt[Sec. 3.3Ð3.4] {Minc:1988}, permutation matrices exist to put $\A$ into a cyclic normal form as in \eqref{eq:CyclicNormal}, and when $\gamma$ is greater than $2$, into a non-symmetric cyclic normal form (with sub-diagonal blocks as can be seen in \eqref {eq:CyclicNormal}), in which case 
\an{\label{eq:PAP}
\P \A \P\tr \neq (\P \A \P\tr)\tr = \P \A\tr \P\tr.
}
But then $\A \neq \A\tr$, contrary to hypothesis.  Therefore either $\gamma = 2$, or $\A$ is primitive.  
\end{proof}

\begin {PropositionNum}\label {Proposition:SimpleGraph}
The adjacency matrix of a connected simple graph is primitive if and only if the graph is not bipartite.
\end {PropositionNum}
\begin{proof}
The adjacency matrix, $\A$, of a connected simple graph is an irreducible sign-symmetric nonnegative matrix, so Proposition \ref {Proposition:PrimitiveOrBipartite} applies. The adjacency matrix $\A$ of a bipartite graph can always be permuted into a cyclic normal form of period $2$, hence its period is always divisible by $2$.  But its period cannot be greater than $2$ because then $\A$ could be permuted into a non-symmetric cyclic normal form, contrary to its symmetry.  Therefore it is cyclic of period $2$ if and only if the graph is bipartite.   By Proposition \ref {Proposition:PrimitiveOrBipartite}, if the adjacency matrix is not bipartite, it is thus primitive.
\end{proof}

\begin{CorollaryNum}\label{Proposition:NotBipartite}
The adjacency matrix of a connected simple graph is \TFI\ if and only if the graph is not bipartite.
\end{CorollaryNum}
\begin{proof}
This is a direct consequence of combining Proposition \ref {Proposition:TIF-Primitive} and Proposition \ref {Proposition:SimpleGraph}.
\end{proof}

\Remark
{
Joel E. Cohen (personal communication) wondered how much of a gap there was between Friedland's \citeauthor{Friedland:1981}'s condition that $\A$ be irreducible and have positive diagonal elements, and the condition found here that $\A$ and $\A\tr\A$ be irreducible.   He found that the gap was small --- only one diagonal element, in the case of an $n$-cycle permutation matrix augmented with positive diagonal elements:  one diagonal element may be set to zero while maintaining the irreducibility of $\A\tr\A$, leaving $2n-1$ positive elements;  but two diagonal elements set to zero make $\A\tr\A$ reducible.  
 
The number $2n-1$ can be seen to derive from the requirement that the bipartite graph of $\A$ be connected.  The graph has $2n$ vertices, and $2n-1$ edges are required to connect them.  Any minimally connected graph must be a tree, since an edge that is part of a cycle may be removed without disconnecting the graph.
}	

The terms 
\emph{indecomposable} \citep{Menon:and:Schneider:1969:Spectrum} \citep[p. 329]{Brualdi:2006:Combinatorial} and 
\emph{chainable} \citep{Hartfiel:and:Maxson:1975:Chainable,Hershkowitz:Rothblum:and:Schneider:1988} 
have been used to refer to a matrix whose associated bipartite graph is connected.   The first use of  `chainable' appears to have been by Sinkhorn and Knopp \citet{Sinkhorn:and:Knopp:1969} for square matrices, and is defined to be the case where, for any two nonzero elements $A_{i_1 j_1}, A_{i_k j_k}$, there is a sequence $A_{i_1 j_1}, \ldots, A_{i_k j_k}$ of nonzero elements satisfying $i_t = i_{t{+}1}$ or $j_t = j_{t{+}1}$.  Hartfiel and Maxson \citet{Hartfiel:and:Maxson:1975:Chainable} modify `chainable' to apply to $(0,1)$-matrices of order $m \times n$, excluding matrices with a row or column of all zeros.  In their Theorem 1.2 they show that the bipartite graph associated with the matrix is connected if and only if the matrix is chainable.  In their Lemma 1.1, they show that $\A$ is chainable if and only if no permutation matrices $\P$ and $\Q$ can produce the block form
\ab{
\P \A \Q = \Bmatr{\A_1 & \0\\ \0 & \A_2}.
}
It is this property that Hershkowitz et al.\ \citet[Definition 2.12]{Hershkowitz:Rothblum:and:Schneider:1988} use to actually define chainable matrices.  In an earlier paper \citep{Menon:and:Schneider:1969:Spectrum}, matrices defined by this property are referred  to as `indecomposable', and this usage is maintained in \citet[p. 340]{Brualdi:2006:Combinatorial}.   `Indecomposable' may cause confusion, however, because a matrix such as \eqref{eq:PartDecompTFI} is then an `indecomposable partly-decomposable matrix'.  `Chainable' is free of this seeming contradiction. (Other similarly confusing terminology remains; for example, that irreducible matrices are completely reducible \citep[p. 127]{Eaves:Hoffman:Rothblum:and:Schneider:1985}).

\begin{PropositionNum}
If $\A$ is two-fold irreducible but not fully indecomposable, then there is no doubly stochastic matrix with the same sign pattern as $\A$.
\end{PropositionNum}
\begin{proof}
This statement is simply a contrapositive of \citet[Lemma 1]{Sinkhorn:and:Knopp:1969}: A nonnegative matrix $\A$ is fully indecomposable if and only if it is chainable and has doubly stochastic pattern.   Hence, if a matrix is not fully indecomposable, one or the other of the two properties must be violated;  since two-fold irreducible matrices are chainable, then the sign pattern must not be doubly stochastic if the matrix is to be partly decomposable.  
\end{proof}

A memorable way to characterize chainable matrices introduced by Sinkhorn and Knopp \citep[p. 68] {Sinkhorn:and:Knopp:1969} is that a path can be made between any two nonzero elements by moving as a rook does in chess from one nonzero element to another.  Irreducible matrices can be characterized in a corresponding way with the following kind of move: 
\begin{PropositionNum}[Board Moves for Irreducibility]
For a square matrix, starting with one nonzero element, let a sequence of nonzero elements be generated using moves with the following structure:
\enumlist{
\item\  move to the reflection of the element's position across the diagonal;\label{item:reflect}
\item\  move horizontally to a nonzero element.\label{item:horizontal}
}
A matrix is irreducible if and only there is a sequence of such moves from any nonzero element to any other nonzero element, and every row and every column has a nonzero element.  Equivalently, move \ref{item:horizontal} may be replaced with all vertical moves.
\end{PropositionNum}
\begin{proof}
Starting from nonzero element $A_{k_1, k_2}$, reflection across the diagonal means going from position $(k_1, k_2)$ to $(k_2, k_1)$.  The horizontal move then takes one from $(k_2, k_1)$ to a nonzero element $A_{k_2, k_3}$ if such exists.  Reflection takes one to $(k_3, k_2)$, and the next horizontal move takes one to a nonzero element $A_{k_3, k_4}$, etc..  The sequence of nonzero elements generated by moves \ref{item:reflect} and \ref{item:horizontal} therefore has the form $(A_{k_1 k_2}, A_{k_2 k_3}, \ldots, A_{k_{p-1} k_p})$.    

Suppose that such a sequence exists from any nonzero $A_{i_1 j_1}$ to any nonzero $A_{i_2 j_2}$ and that every row and every column has at least one nonzero element.  Then for any pair $(i,j)$ there are nonzero elements $A_{i h}$ and $A_{k j}$ for some $h, k \in 1, \ldots, n$.  Since there is a sequence of nonzero elements between $A_{i h}$ and $A_{k j}$, the condition for irreducibility in Theorem \ref{Theorem:Path} is met.

Conversely, suppose that $\A$ is irreducible.  Since there is a sequence of nonzero elements from every $i$ to every $j$, there must be a nonzero element in each row $i$ and each column $j$.  Suppose that $A_{i_1 i_2}$  and $A_{j_1 j_2}$ are nonzero.  By Theorem \ref{Theorem:Path} there is a path of nonzero elements from $i_2$ to $j_1$, $(A_{i_2 k_1}, A_{k_1 k_2}, \ldots, A_{k_p j_1})$.  This path joins $A_{i_1 i_2}$  and $A_{j_1 j_2}$ to create $(A_{i_1 i_2}, A_{i_2 k_1}, A_{k_1 k_2}, \ldots, A_{k_p j_1}, A_{j_1 j_2})$, which shows that any pair of nonzero elements $A_{i_1 i_2}$  and $A_{j_1 j_2}$ can be connected by a sequence of nonzero elements as generated by moves \ref{item:reflect} and \ref{item:horizontal}.

A sequence of moves on $\A$ is equivalent to a sequence of moves on $\A\tr$ where horizontal moves are replaced by vertical moves for step \ref{item:horizontal}.  Since $\A$ is irreducible if and only if $\A\tr$ is irreducible, horizontal moves in step \ref{item:horizontal} may be replaced by vertical moves in the Proposition.
\end{proof}

Useful reviews of the properties of chainable matrices can be found in \citet[Chapter 5]{Sachkov:and:Tarakanov:2002} and \citep[Chapter 8]{Brualdi:2006:Combinatorial}.   Chainable matrices are also to be found under the rubric of `transportation polytopes', and in particular are a means to  characterize the \emph{nondegenerate polytopes} \citep[Chapter 8]{Brualdi:2006:Combinatorial}.  The extreme points of the polytope of square matrices have $2n-1$ positive elements (\citealt [p. 160]{Sachkov:and:Tarakanov:2002}; \citealt[p. 340]{Brualdi:2006:Combinatorial}).

The \TFI\ matrices that appear here are the intersection between the chainable and the irreducible nonnegative matrices.  Further characterization of  this intersection may prove of value.

\section*{Acknowledgements}
Thanks go to Joel E. Cohen for sharing with me the open questions from which this paper grew, and for his comments;  to Roger Nussbaum for his definitive treatment which provided the key tools used here; to Shmuel Friedland for referring Joel Cohen to me; to each for populating the field with so many fruitful results; and to Laura Marie Herrmann for assistance with the literature search.


\begin{thebibliography}{10}

\bibitem{Berman:and:Grone:1988}
{\sc A.~Berman and R.~Grone}, {\em Bipartite completely positive matrices},
  Proceedings of the Cambridge Philosophical Society, 103 (1988), pp.~269--276.

\bibitem{Berman:and:Plemmons:1994}
{\sc A.~Berman and R.~J. Plemmons}, {\em Nonnegative Matrices in the
  Mathematical Sciences}, Society for Industrial and Applied Mathematics
  (SIAM), Philadelphia, 2nd~ed., 1994.

\bibitem{Boche:and:Stanczak:2008:Strict}
{\sc H.~Boche and S.~Stanczak}, {\em Strict convexity of the feasible log-{SIR}
  region}, {IEEE} Transactions on Communications, 56 (2008), pp.~1511--1518.

\bibitem{Brualdi:2006:Combinatorial}
{\sc R.~Brualdi}, {\em Combinatorial Matrix Classes}, vol.~108 of Encyclopedia
  of Mathematics and Its Applications, Cambridge University Press, Cambridge,
  2006.

\bibitem{Brualdi:and:Ryser:1991}
{\sc R.~Brualdi and H.~Ryser}, {\em Combinatorial Matrix Theory}, vol.~39 of
  Encyclopedia of Mathematics and its Applications, Cambridge University Press,
  Cambridge, 1991.

\bibitem{Cohen:1981:Convexity}
{\sc J.~E. Cohen}, {\em Convexity of the dominant eigenvalue of an essentially
  nonnegative matrix}, Proceedings of the American Mathematical Society, 81
  (1981), pp.~657--658.

\bibitem{Cohen:2012:Cauchy}
\leavevmode\vrule height 2pt depth -1.6pt width 23pt, {\em Cauchy inequalities
  for the spectral radius of products of diagonal and nonnegative matrices}.
\newblock Submitted.
  \href{mailto:cohen@rockefeller.edu?subject=Cauchy-Inequalities}
  {cohen@rockefeller.edu}, 2012.

\bibitem{Cohen:Friedland:Kato:and:Kelly:1982}
{\sc J.~E. Cohen, S.~Friedland, T.~Kato, and F.~P. Kelly}, {\em Eigenvalue
  inequalities for products of matrix exponentials}, Linear Algebra and Its
  Applications, 45 (1982), pp.~55--95.

\bibitem{Eaves:Hoffman:Rothblum:and:Schneider:1985}
{\sc B.~Eaves, A.~Hoffman, U.~Rothblum, and H.~Schneider}, {\em
  Line-sum-symmetric scalings of square nonnegative matrices}, Mathematical
  Programming Study, 25 (1985), pp.~124--141.

\bibitem{Friedland:1981}
{\sc S.~Friedland}, {\em Convex spectral functions}, Linear and Multilinear
  Algebra, 9 (1981), pp.~299--316.

\bibitem{Gantmacher:1959vol2}
{\sc F.~R. Gantmacher}, {\em The Theory of Matrices}, vol.~2, Chelsea
  Publishing Company, New York, 1959.

\bibitem{Hajnal:1958:Weak}
{\sc J.~Hajnal}, {\em Weak ergodicity in non-homogeneous markov chains},
  Proceedings of the Cambridge Philosophical Society, 54 (1958), pp.~233--246.

\bibitem{Hartfiel:and:Maxson:1975:Chainable}
{\sc D.~Hartfiel and C.~Maxson}, {\em The chainable matrix, a special
  combinatorial matrix}, Discrete Mathematics, 12 (1975), pp.~245--256.

\bibitem{Hershkowitz:Rothblum:and:Schneider:1988}
{\sc D.~Hershkowitz, U.~Rothblum, and H.~Schneider}, {\em Classifications of
  nonnegative matrices using diagonal equivalence}, SIAM Journal of Matrix
  Analysis and Applications, 9 (1988), pp.~455--460.

\bibitem{Hogben:2007}
{\sc L.~Hogben}, ed., {\em Handbook of Linear Algebra}, Chapman and Hall, 2007.

\bibitem{Kantrowitz:Neumann:and:Ransford:2011}
{\sc R.~Kantrowitz, M.~Neumann, and T.~Ransford}, {\em Regularity, scrambling,
  and the steady state for stochastic matrices}, in Function Spaces in Modern
  Analysis: Sixth Conference on Function Spaces, May 18-22, 2010, Southern
  Illinois University, Edwardsville, K.~Jarosz, ed., vol.~547, Providence, RI,
  2011, American Mathematical Society, pp.~153--164.

\bibitem{Kingman:1961:Convexity}
{\sc J.~F.~C. Kingman}, {\em {A convexity property of positive matrices}}, The
  Quarterly Journal of Mathematics, 12 (1961), pp.~283--284.

\bibitem{Menon:and:Schneider:1969:Spectrum}
{\sc M.~Menon and H.~Schneider}, {\em The spectrum of a nonlinear operator
  associated with a matrix}, Linear Algebra Appl, 2 (1969), pp.~321--334.

\bibitem{Minc:1988}
{\sc H.~Minc}, {\em Nonnegative Matrices}, John Wiley and Sons, New York, 1988.

\bibitem{Nussbaum:1986:Convexity}
{\sc R.~D. Nussbaum}, {\em Convexity and log convexity for the spectral
  radius}, Linear Algebra and Its Applications, 73 (1986), pp.~59--122.

\bibitem{Sachkov:and:Tarakanov:2002}
{\sc V.~Sachkov and V.~Tarakanov}, {\em Combinatorics of nonnegative matrices},
  vol.~213, Amer Mathematical Society, 2002.

\bibitem{Seneta:2006}
{\sc E.~Seneta}, {\em Non-negative Matrices and {Markov} Chains},
  Springer-Verlag, New York, revised~ed., 2006.

\bibitem{Sinkhorn:and:Knopp:1969}
{\sc R.~Sinkhorn and P.~Knopp}, {\em Problems involving diagonal products in
  nonnegative matrices}, Transactions of the American Mathematical Society, 136
  (1969), pp.~67--75.

\end{thebibliography}
\end{document}